\documentclass[10pt]{article}
\usepackage{graphicx,url}
\usepackage{dcolumn}
\usepackage{bm}
\usepackage{amsmath,amsfonts,amssymb,color}

\newtheorem{definition}{\noindent{\it Definition}}[section]
\newtheorem{theorem}{\noindent{\it Theorem}}[section]
\newtheorem{proposition}{\noindent{\it Proposition}}[section]
\newtheorem{lemma}{\noindent{\it Lemma}}[section]
\newtheorem{remark}{\noindent{\it Remark}}[section]
\newtheorem{corollary}{\noindent{\it Corollary}}[section]
\newtheorem{example}{\noindent{\it Example}}[section]
\newenvironment{proof}{\noindent{\it Proof:}}{$\hfill$ $\Box$\\ }

\begin{document}

\title{On semi-vector spaces and semi-algebras}

\author{Giuliano G. La Guardia, Jocemar de Q. Chagas, Ervin K. Lenzi, Leonardo Pires
\thanks{Giuliano G. La Guardia ({\tt \small
gguardia@uepg.br}), Jocemar de Q. Chagas ({\tt \small jocemarchagas@uepg.br})
and Leonardo Pires ({\tt \small lpires@uepg.br}) are
with Department of Mathematics and Statistics,
State University of Ponta Grossa (UEPG), 84030-900, Ponta Grossa -
PR, Brazil. Ervin K. Lenzi ({\tt \small eklenzi@uepg.br}) is with Department of Physics,
State University of Ponta Grossa (UEPG), 84030-900, Ponta Grossa -
PR, Brazil. Corresponding author: Giuliano G. La Guardia ({\tt \small
gguardia@uepg.br}). }}

\maketitle

\begin{abstract}
It is well-known that the theories of semi-vector spaces and semi-algebras --
which were not much studied over time -- are utilized/applied in Fuzzy Set
Theory in order to obtain extensions of the concept of fuzzy numbers as well as to
provide new mathematical tools to investigate properties and new results
on fuzzy systems. In this paper we investigate the theory of
semi-vector spaces over the semi-field of nonnegative real numbers
${\mathbb R}_{0}^{+}$. We prove several results concerning semi-vector
spaces and semi-linear transformations.
Moreover, we introduce in the literature the concept of
eigenvalues and eigenvectors of a semi-linear operator, describing in some cases
how to compute them. Topological properties of semi-vector spaces
such as completeness and separability are also investigated. New families of
semi-vector spaces derived from semi-metric,
semi-norm, semi-inner product, among others are exhibited. Additionally,
some results on semi-algebras are presented.
\end{abstract}

\emph{keywords}: semi-vector space; semi-algebras; semi-linear operators

\section{Introduction}

The concept of semi-vector space was introduced by Prakash and Sertel in
\cite{Prakash:1974}. Roughly speaking, semi-vector spaces are
``vector spaces" where the scalars are in a semi-field. Although the concept of
semi-vector space was investigated over time,
there exist few works available in the literature dealing
with such spaces \cite{Radstrom:1952,Prakash:1974,Prakash:1976,Pap:1980,Gahler:1999,Janyska:2007,Milfont:2021}.
This fact occurs maybe due to the limitations that such concept brings,
i.e., the non-existence of (additive) symmetric for some (for all) semi-vector. A
textbook in such a topic of research is the book by Kandasamy~\cite{Kandasamy:2002}.

Although the seminal paper on semi-vector spaces is \cite{Prakash:1974},
the idea of such a concept was implicit in \cite{Radstrom:1952},
where Radstrom shown that a semi-vector space over the semi-field of nonnegative real
numbers can be extended to a real vector space
(see \cite[Theorem 1-B.]{Radstrom:1952}).
In \cite{Prakash:1974}, Prakash and Sertel investigated the
structure of topological semi-vector spaces. The authors were concerned with
the study of the existence of fixed points in compact convex sets and also
to generate min-max theorems in topological semi-vector spaces. In
\cite{Prakash:1976}, Prakash and Sertel investigated
properties of the topological semi-vector space consisting
of nonempty compact subsets of a real Hausdorff topological vector space.
In \cite{Pap:1980}, Pap investigated and formulated the concept of integrals of
functions having, as counter-domain, complete semi-vector spaces. W.
Gahler and S. Gahler \cite{Gahler:1999} showed that a (ordered) semi-vector
space can be extended to a (ordered) vector space and a (ordered)
semi-algebra can be extended to a (ordered) algebra. Moreover, they
provided an extension of fuzzy numbers. Janyska et al.~\cite{Janyska:2007}
developed such theory (of semi-vector space) by proving useful results
and defining the semi-tensor product of (semi-free) semi-vector spaces.
They were also interested to propose an algebraic model of physical scales.
Canarutto~\cite{Canarutto:2012} explored the concept of semi-vector spaces
to express aspects and to exploit
nonstandard mathematical notions of basics of quantum particle physics on
a curved Lorentzian background. Moreover, he dealt with the case of
electroweak interactions. Additionally, in \cite{Canarutto:2016},
Canarutto provided a suitable formulation of the fundamental mathematical
concepts with respect to quantum field theory. Such a paper presents a
natural application of the concept of semi-vector spaces and semi-algebras.
Recently, Bedregal et al. \cite{Milfont:2021} investigated (ordered)
semi-vector spaces over a weak semi-field $K$ (i.e., both $(K, +)$ and
$(K, \bullet)$ are monoids) in the context of fuzzy sets and applying
the results in multi-criteria group decision-making.

In this paper we extend the theory of semi-vector spaces. The semi-field of
scalars considered here is the semi-field of nonnegative real numbers.
We prove several results in the context of semi-vector spaces and semi-linear
transformations. We introduce the concept of semi-eigenvalues
and semi-eigenvectors of an operator and of a matrix, showing how to compute
it in specific cases. We investigate topological properties such as
completeness, compactness and separability of semi-vector spaces.
Additionally, we present interesting new families
of semi-vector spaces derived from semi-metric, semi-norm, semi-inner
product, metric-preserving
functions among others. Furthermore, we show some results concerning semi-algebras.
Summarizing, we provide new results on semi-vector spaces and semi-algebras,
although such theories are very difficult to be investigated due to the fact that vectors do
not even have (additive) symmetrical. These new results can be possibly utilized in the 
theory of fuzzy sets in order to extend it or in the generation of new results 
concerning such a theory.

The paper is organized as follows. In Section~\ref{sec2} we recall some concepts
on semi-vector spaces which will be utilized in this work.
In Section~\ref{sec3} we present and prove several results concerning semi-vector spaces
and semi-linear transformations. We introduce naturally the concepts of eigenvalue
and eigenvector of a semi-linear operator and of matrices. Additionally, we exhibit
and show interesting examples of semi-vector spaces derived from semi-metric, semi-norms,
metric-preserving functions among others. Results concerning semi-algebras are
also presented. In Section~\ref{sec3a} we show relationships between Fuzzy Set Theory
and the theory of semi-vector spaces and semi-algebras. Finally, a summary
of this paper is presented in Section~\ref{sec4}.

\section{Preliminaries}\label{sec2}

In this section we recall important facts on semi-vector spaces
necessary for the development of this work. In order to
define formally such concept, it is necessary to define the concepts of
semi-ring and semi-field.

\begin{definition}\label{defSR}
A semi-ring $(S, + , \bullet )$ is a set $S$ endowed with two binary
operations, $+: S\times S\longrightarrow S$ (addition),
$\bullet: S\times S\longrightarrow S$ (multiplication) such that: $\operatorname{(1)}$
$(S, +)$ is a commutative monoid; $\operatorname{(2)}$ $(S, \bullet)$ is a semigroup;
$\operatorname{(3)}$ the multiplication $\bullet$ is distributive with respect to $+$:
$\forall \ x, y, z \in S$, $(x + y)\bullet z = x\bullet z + y\bullet z$ and
$x\bullet(y + z ) = x\bullet y + x\bullet z$.
\end{definition}

We write $S$ instead of writing $(S, + , \bullet )$ if there is not
possibility of confusion. If the multiplication $\bullet$
is commutative then $S$ is a commutative semi-ring. If there exists $1
\in S$ such that, $ \forall \ x \in S$ one has
$1\bullet x = x\bullet 1 = x$, then $S$ is a semi-ring with identity.

\begin{definition}\cite[Definition 3.1.1]{Kandasamy:2002}\label{defSF}
A semi-field is an ordered triple $(K, +, \bullet )$ which is a
commutative semi-ring with unit satisfying the following conditions:
$\operatorname{(1)}$ $\forall \ x, y \in K$, if $x+y=0$ then $x=y=0$;
$\operatorname{(2)}$ if $x, y \in K$ and $x\bullet y = 0$ then $x=0$ or $y=0$.
\end{definition}

Before proceeding further, it is interesting to observe that in \cite{Gahler:1999}
the authors considered the additive cancellation law in the definition of
semi-vector space. In \cite{Janyska:2007}, the authors did not assume the existence
of the zero (null) vector.

In this paper we consider the definition of a semi-vector space in the context of
that shown in \cite{Gahler:1999}, Sect.3.1.

\begin{definition}\label{defSVS}
A semi-vector space over a semi-field $K$ is a ordered triple
$(V,$ $+, \cdot)$, where $V$ is a set endowed with the operations
$+: V\times V\longrightarrow V$ (vector addition) and
$\cdot:  K\times V\longrightarrow V$
(scalar multiplication) such that:
\begin{itemize}
\item [ $\operatorname{(1)}$] $(V, +)$ is an abelian monoid
equipped with the additive cancellation law: $\forall \ u, v, w \in V$,
if $u + v = u + w$ then $v = w$;
\item [ $\operatorname{(2)}$] $\forall$ $\alpha\in K$ and $\forall$
$u, v \in V$, $\alpha (u+v)=\alpha u + \beta v$;
\item [ $\operatorname{(3)}$] $\forall$ $\alpha, \beta \in K$
and $\forall$ $v\in V$, $(\alpha + \beta)v= \alpha v + \beta v$;
\item [ $\operatorname{(4)}$] $\forall$ $\alpha, \beta \in K$
and $\forall$ $v\in V$, $(\alpha\beta)v=\alpha (\beta v)$;
\item [ $\operatorname{(5)}$] $\forall$ $v \in V$ and $1 \in K$, $1v=v$.
\end{itemize}
\end{definition}

Note that from Item~$\operatorname{(1)}$ of Definition~\ref{defSVS},
all semi-vector spaces considered in this paper are \emph{regular},
that it, the additive cancellation law is satisfied. The zero (or null) vector
of $V$, which is unique, will be denoted by $0_{V}$. Let $v \in V$, $v\neq 0 $.
If there exists $u \in V$ such
that $v + u =0$ then $v$ is said to be \emph{symmetrizable}.
A semi-vector space $V$ is said to be \emph{simple} if the unique
symmetrizable element is the zero vector $0_{V}$. In other words, $V$ is
simple if it has none nonzero symmetrizable elements.

\begin{definition}\cite[Definition 1.4]{Janyska:2007}\label{defSBasis}
Let $V$ be a simple semi-vector space over ${\mathbb R}_{0}^{+}$.
A subset $B \subset V$ is called a semi-basis of $V$ if
every $v \in V$, $v\neq 0$, can be written in a unique way as $v =
\displaystyle\sum_{i \in I_v}^{} v^{(i)} b_i$,
where $v^{(i)} \in {\mathbb R}^{+}$, $b_i \in B$ and $I_v$ is a
finite family of indices uniquely determined by $v$. The finite
subset $B_v \subset B$ defined by $B_v := \{b_i \}_{i \in I_v }$
is uniquely determined by $v$. If a semi-vector
space $V$ admits a semi-basis then it is said to be semi-free.
\end{definition}

The concept of semi-dimension can be defined in analogous way to
semi-free semi-vector spaces due to the next result.

\begin{corollary}\cite[Corollary 1.7]{Janyska:2007}
Let $V$ be a semi-free semi-vector space. Then all semi-bases of $V$ have
the same cardinality.
\end{corollary}
Therefore, the semi-dimension of a semi-free semi-vector space is the
cardinality of a semi-basis (consequently, of all semi-bases) of $V$.
We next present some examples of semi-vector spaces.

\begin{example}\label{ex1}
All real vector spaces are semi-vector spaces, but they are not simple.
\end{example}

\begin{example}\label{ex2}
The set ${[{\mathbb R}_{0}^{+}]}^{n}=\underbrace{{\mathbb R}_{0}^{+}
\times \ldots \times {\mathbb R}_{0}^{+}}_{n \operatorname{times}}$
endowed with the usual sum of coordinates and scalar multiplication is a
semi-vector space over ${\mathbb R}_{0}^{+}$.
\end{example}

\begin{example}\label{ex3}
The set ${\mathcal M}_{n\times m}({\mathbb R}_{0}^{+})$ of matrices $n \times m$
whose entries are nonnegative real numbers equipped with the sum of
matrices and multiplication of a matrix by a scalar (in ${\mathbb R}_{0}^{+}$, of course)
is a semi-vector space over ${\mathbb R}_{0}^{+}$.
\end{example}

\begin{example}\label{ex4}
The set ${\mathcal P}_{n}[x]$ of polynomials with coefficients
from ${\mathbb R}_{0}^{+}$ and degree less than or equal
to $n$, equipped with the usual of polynomial sum and scalar multiplication,
is a semi-vector space.
\end{example}

\begin{definition}\label{semi-subspace}
Let $(V, +, \cdot )$ be a semi-vector space over ${\mathbb R}_{0}^{+}$. We say that a
non-empty subset $W$ of $V$ is a semi-subspace of $V$ if $W$ is closed
under both addition and scalar multiplication of $V$, that is,
\begin{itemize}
\item [ $\operatorname{(1)}$] $\forall \ w_1 , w_2 \in W \Longrightarrow
w_1 + w_2 \in W$;
\item [ $\operatorname{(2)}$] $\forall \ \lambda \in {\mathbb R}_{0}^{+}$ and
$\forall \ w \in W \Longrightarrow \lambda w \in W$.
\end{itemize}
\end{definition}

The uniqueness of the zero vector implies that for each $\lambda \in
{\mathbb R}_{0}^{+}$ on has $\lambda 0_{V} = 0_{V}$. Moreover, if $ v \in V$,
it follows that $0 v = 0 v + 0 v$; applying the regularity one obtains $0 v =0_{V}$.
Therefore, from Item~$\operatorname{(2)}$, every
semi-subspace contains the zero vector.

\begin{example}\label{ex4a}
Let ${\mathbb Q}_{0}^{+}$ denote the set of nonnegative rational numbers.
The semi-vector space ${\mathbb Q}_{0}^{+}$ considered as an
${\mathbb Q}_{0}^{+}$ space is a semi-subspace of
${\mathbb R}_{0}^{+}$ considered as an ${\mathbb Q}_{0}^{+}$ space.
\end{example}

\begin{example}\label{ex4b}
For each positive integer $ i \leq n$, the subset ${\mathcal P}_{(i)}[x]\cup \{0_{p}\}$,
where ${\mathcal P}_{(i)}[x]=\{p(x); \partial (p(x))=i \} $
and $0_{p}$ is the null polynomial, is a semi-subspace of ${\mathcal P}_{n}[x]$, shown in
Example~\ref{ex4}.
\end{example}

\begin{example}\label{ex4c}
The set of diagonal matrices of order $n$ with entries in ${\mathbb R}_{0}^{+}$ is a
semi-subspace of ${\mathcal M}_{n}({\mathbb R}_{0}^{+})$,
where the latter is the semi-vector space of square matrices with entries in ${\mathbb R}_{0}^{+}$
(according to Example~\ref{ex3}).
\end{example}

\begin{definition}\cite[Definition 1.22]{Janyska:2007}\label{semilineartrans}
Let $V$ and $W$ be two semi-vector spaces and $T: V\longrightarrow W$ be a map.
We say that $T$ is a semi-linear transformation if:
$\operatorname{(1)}$ $\forall \ v_1, v_2 \in V$, $T(v_1 + v_2) = T(v_1) + T(v_2)$;
$\operatorname{(2)}$ $\forall \lambda \in {\mathbb R}_{0}^{+}$ and
$\forall \ v \in V$, $T(\lambda v) =\lambda T(v)$.
\end{definition}

If $U$ and $V$ are semi-vector spaces then the set
$\operatorname{Hom}(U, V)=\{ T:U\longrightarrow V;
T \operatorname{is \ semi-linear} \}$ is also a semi-vector space.

\section{The New Results}\label{sec3}

We start this section with important remarks.

\begin{remark}\label{mainremark}
\begin{itemize}
\item [ $\operatorname{(1)}$] Throughout this section we always consider
that the semi-field $K$ is the set of nonnegative real numbers, i.e.,
$K= {\mathbb R}_{0}^{+}={\mathbb R}^{+}\cup \{0\}$.

\item [ $\operatorname{(2)}$] In the whole section (except
Subsection~\ref{subsec2}) we assume that the semi-vector
spaces $V$ are simple, i.e., the unique symmetrizable element is the zero vector $0_{V}$.

\item [ $\operatorname{(3)}$] It is well-known that a semi-vector space
$(V, +, \cdot)$ can be always extended to a vector space according
to the equivalence relation on $V \times V$ defined by $(u_1 , v_1 ) \sim (u_2 , v_2 )$
if and only if $u_1 + v_2 = v_1 + u_2$ (see \cite{Radstrom:1952}; see also
\cite[Section 3.4]{Gahler:1999}). However, our results are obtained without utilizing
such a natural embedding. In other words, if one want to compute, for instance,
the eigenvalues of a matrix defined over ${\mathbb R}_{0}^{+}$ we cannot solve the
problem in the associated vector spaces and then discard the negative ones. Put differently,
all computations performed here are restricted to nonnegative real numbers and
also to the fact that none vector (with exception of $0_V$) has (additive) symmetrical.
However, we will show that, even in this case, several results can be obtained.
\end{itemize}
\end{remark}

\begin{proposition}\label{prop1}
Let $V$ be a semi-vector space over ${\mathbb R}_{0}^{+}$. Then the following hold:
\begin{itemize}
\item [ $\operatorname{(1)}$] let $ v \in V$, $ v \neq 0_{V}$, and $\lambda
\in {\mathbb R}_{0}^{+}$; if $\lambda v = 0_{V}$ then $\lambda = 0$;
\item [ $\operatorname{(2)}$] if $\alpha , \beta \in {\mathbb R}_{0}^{+}$,
$v \in V$ and $ v \neq 0_{V}$, then the equality $\alpha v = \beta v$
implies that $\alpha = \beta$.
\end{itemize}
\end{proposition}
\begin{proof}
$\operatorname{(1)}$ If $\lambda \neq 0$ then there exists its
multiplicative inverse ${\lambda}^{-1}$, hence
$ 1 v = {\lambda}^{-1} 0_{V}= 0_{V}$, i.e., $v = 0_{V}$, a contradiction.\\
$\operatorname{(2)}$ If $\alpha \neq \beta$, assume w.l.o.g.
that $\alpha > \beta$, i.e., there exists a positive real
number $c$ such that $\alpha = \beta + c$. Thus, $\alpha v = \beta v$ implies
$\beta v + c v = \beta v$. From the cancellation law we have $c v = 0_{V}$,
and from Item~$\operatorname{(1)}$ it follows that $c = 0$, a contradiction.
\end{proof}

We next introduce in the literature the concept of eigenvalue and eigenvector of a
semi-linear operator.

\begin{definition}\label{eigenvector}
Let $V$ be a semi-vector space and $T:V\longrightarrow V$ be a semi-linear
operator. If there exist a non-zero vector $v \in V$ and
a nonnegative real number $\lambda$ such that $T(v)=\lambda v$, then $\lambda$ is an
eigenvalue of $T$ and $v$ is an eigenvector of $T$ associated with $\lambda$.
\end{definition}

As it is natural, the zero vector joined to the set of the eigenvectors
associated with a given eigenvalue has a semi-subspace structure.

\begin{proposition}\label{eigenspace}
Let $V$ be a semi-vector space over ${\mathbb R}_{0}^{+}$ and
$T:V\longrightarrow V$ be a semi-linear operator. Then the set
$V_{\lambda} = \{ v \in V  ; T(v)=\lambda v \}\cup \{0_{V}\}$ is a semi-subspace of $V$.
\end{proposition}
\begin{proof}
From hypotheses, $V_{\lambda}$ is non-empty. Let $u, v \in V_{\lambda}$, i.e.,
$T(u)=\lambda u $ and $T(v)=\lambda v $. Hence, $T(u + v )= T(u) + T(v)=
\lambda (u + v )$, i.e., $u + v \in V_{\lambda}$. Further, if $\alpha \in {\mathbb R}_{0}^{+}$
and $u \in V$, it follows that $T(\alpha u)=\alpha T(u)= \lambda (\alpha u)$, that is,
$\alpha u \in V_{\lambda}$. Therefore, $V_{\lambda}$ is a semi-subspace of $V$.
\end{proof}

The next natural step would be to introduce the characteristic polynomial
of a matrix, according to the standard Linear Algebra. However, how
to compute $\det (A -\lambda I)$ if $-\lambda$ can be a negative real number? Based
on this fact we must be careful to compute the eigenvectors of a matrix.
In fact, the main tools to be utilized in computing eigenvalues/eigenvectors of
a square matrix whose entries are nonnegative real numbers is the
additive cancellation law in ${\mathbb R}_{0}^{+}$ and also the fact that positive
real numbers have multiplicative inverse. However, in much cases, such a tools are
not sufficient to solve the problem. Let us see some cases when it is
possible to compute eigenvalues/eigenvectors of a matrix.

\begin{example}\label{examatr1}
Let us see how to obtain (if there exists) an eigenvalue/eigenvector of a
diagonal matrix $A \in {\mathcal M}_{2}({\mathbb R}_{0}^{+})$,
\begin{eqnarray*}
A= \left[\begin{array}{cc}
a & 0\\
0 & b\\
\end{array}
\right],
\end{eqnarray*}
where $a \neq b$ not both zeros.

Let us assume first that $a, b > 0$.
Solving the equation $A v = \lambda v$, that is,
\begin{eqnarray*}
\left[\begin{array}{cc}
a & 0\\
0 & b\\
\end{array}
\right]
\left[\begin{array}{c}
x\\
y\\
\end{array}
\right]=  \left[\begin{array}{c}
\lambda x\\
\lambda y\\
\end{array}
\right],
\end{eqnarray*}
we obtain $\lambda = a$ with associated eigenvector $x(1, 0)$ and
$\lambda = b$ with associated eigenvector $y(0, 1)$.

If $a\neq 0$ and $b = 0$, then $\lambda = a$ with eigenvectors $x(1, 0)$.

If $a = 0$ and $b \neq 0$, then $\lambda = b$ with eigenvectors $y(0, 1)$.
\end{example}

\begin{example}\label{examatr2}
Let $A \in {\mathcal M}_{2}({\mathbb R}_{0}^{+})$ be a matrix of the form
\begin{eqnarray*}
A= \left[\begin{array}{cc}
a & b\\
0 & a\\
\end{array}
\right],
\end{eqnarray*}
where $a \neq b$ are positive real numbers. Let us solve the matrix equation:

\begin{eqnarray*}
\left[\begin{array}{cc}
a & b\\
0 & a\\
\end{array}
\right]
\left[\begin{array}{c}
x\\
y\\
\end{array}
\right]=  \left[\begin{array}{c}
\lambda x\\
\lambda y\\
\end{array}
\right].
\end{eqnarray*}
If $ y \neq 0$, $\lambda = a$; hence $b y = 0$, which
implies $b=0$, a contradiction. If $ y = 0$, $x \neq 0$; hence $\lambda = a$
with eigenvectors $(x, 0)$.
\end{example}

If $V$ and $W$ are semi-free semi-vector spaces then it is possible to define the
matrix of a semi-linear transformation $T: V \longrightarrow W$ as in the usual
case (vector spaces).

\begin{definition}\label{semi-free matrix}
Let $T: V \longrightarrow W$ be a semi-liner transformation between
semi-free semi-vector spaces with semi-basis $B_1$ and $B_2$, respectively.
Then the matrix $[T]_{B_1}^{B_2}$ is the matrix of the transformation $T$.
\end{definition}

\begin{theorem}\label{diagonalmatrix}
Let $V$ be a semi-free semi-vector space over ${\mathbb R}_{0}^{+}$ and
let $T:V\longrightarrow V$ be a semi-linear operator. Then $T$
admits a semi-basis $B = \{ v_1 , v_2 , \ldots , v_n \}$ such that ${[T]}_{B}^{B}$
is diagonal if and only if $B$ consists of eigenvectors of $T$.
\end{theorem}
\begin{proof}
The proof is analogous to the case of vector spaces.
Let $B=\{ v_1 , v_2 , \ldots ,$ $v_n \}$ be a semi-basis of $V$ whose elements
are eigenvectors of $T$. We then have:
\begin{eqnarray*}
T(v_1)= {\lambda}_1 v_1 + 0 v_2 + \ldots + 0 v_n,\\
T(v_2)= 0 v_1 + {\lambda}_{2} v_2 + \ldots + 0 v_n,\\
\vdots\\
T(v_n)= 0 v_1 + 0 v_2 + \ldots + {\lambda}_{n} v_n,
\end{eqnarray*}
which implies that $[T]_{B}^{B}$ is of the form
\begin{eqnarray*}
[T]_{B}^{B}= \left[\begin{array}{ccccc}
{\lambda}_1 & 0 & 0 & \ldots & 0\\
0 & {\lambda}_2 & 0 & \ldots & 0\\
\vdots & \vdots & \vdots & \ldots & \vdots\\
0 & 0 & 0 & \ldots & {\lambda}_{n}\\
\end{array}
\right].
\end{eqnarray*}
On the other hand, let $B^{*}= \{ w_1 , w_2 , \ldots , w_n \}$ be a
semi-basis of $V$ such that $[T]_{B^{*}}^{B^{*}}$ is diagonal:
\begin{eqnarray*}
[T]_{B^{*}}^{B^{*}}=\left[\begin{array}{ccccc}
{\alpha}_1 & 0 & 0 & \ldots & 0\\
0 & {\alpha}_2 & 0 & \ldots & 0\\
\vdots & \vdots & \vdots & \ldots & \vdots\\
0 & 0 & 0 & \ldots & {\alpha}_{n}\\
\end{array}
\right];
\end{eqnarray*} thus,\\
\begin{eqnarray*}
T(w_1)= {\alpha}_1 w_1 + 0 w_2 + \ldots + 0 w_n = {\alpha}_1 w_1,\\
T(w_2)= 0 w_1 + {\alpha}_{2} w_2 + \ldots + 0 w_n = {\alpha}_{2} w_2,\\
\vdots\\
T(w_n)= 0 w_1 + 0 w_2 + \ldots + {\alpha}_{n} w_n = {\alpha}_{2} w_{n}.
\end{eqnarray*}
This means that $w_i$ are eigenvectors of $T$ with corresponding
eigenvalues ${\alpha}_{i}$, for all $i = 1, 2, \ldots , n$.
\end{proof}

\begin{definition}\label{kernel}
Let $T: V \longrightarrow W$ be a semi-linear transformation. The set
$\operatorname{Ker}(T)=\{ v \in V ; T(v)=0\}$ is called kernel of $T$.
\end{definition}

\begin{proposition}\label{subkernel}
Let $T: V \longrightarrow W$ be a semi-linear transformation.
Then the following hold:
\begin{itemize}
\item [ $\operatorname{(1)}$] $\operatorname{Ker}(T)$ is a semi-subspace of $V$;
\item [ $\operatorname{(2)}$] if $T$ is injective then
$\operatorname{Ker}(T) = \{0_{V}\}$;
\item [ $\operatorname{(3)}$] if $V$ has semi-dimension $1$ then
$\operatorname{Ker}(T) = \{0_{V}\}$ implies that $T$ is injective.
\end{itemize}
\end{proposition}
\begin{proof}
$\operatorname{(1)}$ We have $T(0_{V})= T(0_{V})+T(0_{V})$. Since $W$
is regular, it follows that $T(0_{V})=0_{W}$, which implies
$\operatorname{Ker}(T) \neq \emptyset$. If $u, v \in \operatorname{Ker}(T)$ and $\lambda \in
{\mathbb R}_{0}^{+}$, then $u + v \in \operatorname{Ker}(T)$ and $\lambda v \in
\operatorname{Ker}(T)$, which implies that $\operatorname{Ker}(T)$
is a semi-subspace of $V$.\\
$\operatorname{(2)}$ Since $T(0_{V})=0_{W}$, it follows that
$\{0_{V}\}\subseteq \operatorname{Ker}(T)$. On the other hand, let $ u \in
\operatorname{Ker}(T)$, that is, $T(u)=0_{W}$. Since $T$ is injective, one has $u = 0_{V}$.
Hence, $\operatorname{Ker}(T) = \{0_{V}\}$.\\
$\operatorname{(3)}$ Let $B=\{ v_0 \}$ be a semi-basis of $V$. Assume that
$T(u) = T(v)$, where $u, v \in V$ are such that $u = \alpha v_0$ and $v = \beta v_0 $. Hence,
$\alpha T(v_0) = \beta T(v_0 )$. Since $\operatorname{Ker}(T) = \{0_{V}\}$ and $v_0 \neq 0$,
it follows that $T(v_0) \neq 0$. From Item~$\operatorname{(2)}$ of Proposition~\ref{prop1},
one has $\alpha = \beta$, i.e., $u = v$.
\end{proof}

\begin{definition}\label{image}
Let $T: V \longrightarrow W$ be a semi-linear transformation.
The image of $T$ is the set of all vectors $w \in W$ such
that there exists $v \in V$ with $T(v)=w$, that is, $\operatorname{Im}(T)=\{
w \in W ; \exists \ v \in V \operatorname{with} T(v)=w\}$.
\end{definition}

\begin{proposition}\label{subImage}
Let $T: V \longrightarrow W$ be a semi-linear transformation.
Then the image of $T$ is a semi-subspace of $W$.
\end{proposition}
\begin{proof}
The set $\operatorname{Im}(T)$ is non-empty because $T(0_{V})=0_{W}$. It is
easy to see that if $w_1 , w_2 \in \operatorname{Im}(T)$ and $\lambda \in
{\mathbb R}_{0}^{+}$, then $ w_1 + w_2 \in \operatorname{Im}(T)$ and
$\lambda w_1 \in \operatorname{Im}(T)$.
\end{proof}

\begin{theorem}\label{isosemi}
Let $V$ be a $n$-dimensional semi-free semi-vector space over
${\mathbb R}_{0}^{+}$. Then $V$ is isomorphic to $({\mathbb R}_{0}^{+})^{n}$.
\end{theorem}
\begin{proof}
Let $B = \{ v_1 , v_2 , \ldots , v_n \}$ be a semi-basis of $V$ and consider
the canonical semi-basis
$e_i = (0, 0, \ldots , $ $0, \underbrace{1}_{i}, 0, \ldots, 0)$
of $({\mathbb R}_{0}^{+})^{n}$, where $i=1, 2, \ldots , n$.
Define the map $T:V \longrightarrow ({\mathbb R}_{0}^{+})^{n}$ as follows:
for each $v = \displaystyle\sum_{i=1}^{n}a_i v_i \in V$, put
$T(v) = \displaystyle\sum_{i=1}^{n}a_i e_i$. It is easy to see that $T$
is bijective semi-linear transformation, i.e., $V$ is isomorphic to
$({\mathbb R}_{0}^{+})^{n}$, as required.
\end{proof}

\subsection{Complete Semi-Vector Spaces}\label{subsec1}

We here define and study complete semi-vector spaces, i.e.,
semi-vector spaces whose norm (inner product) induces a metric under
which the space is complete.

\begin{definition}\label{semiBanach}
Let $V$ be a semi-vector space over ${\mathbb R}_{0}^{+}$. If there exists a
norm $\| \ \|:V \longrightarrow {\mathbb R}_{0}^{+}$ on $V$ we say that $V$ is a
normed semi-vector space (or normed semi-space, for short). If the norm defines a metric
on $V$ under which $V$ is complete then $V$ is said to be Banach semi-vector space.
\end{definition}

\begin{definition}\label{semiHilbert}
Let $V$ be a semi-vector space over ${\mathbb R}_{0}^{+}$. If there
exists an inner product $\langle \ , \ \rangle:V\times V
\longrightarrow {\mathbb R}_{0}^{+}$ on $V$ then $V$ is an inner product
semi-vector space (or inner product semi-space).
If the inner product defines a metric on $V$ under which
$V$ is complete then $V$ is said to
be Hilbert semi-vector space.
\end{definition}

The well-known norms on ${\mathbb R}^n$ are also norms on
$[{\mathbb R}_{0}^{+}]^{n}$, as we show in the next propositions.

\begin{proposition}\label{R+1}
Let $V = [{\mathbb R}_{0}^{+}]^{n}$ be the Euclidean semi-vector space
(over ${\mathbb R}_{0}^{+}$) of
semi-dimension $n$ . Define the function $\| \ \|:V \longrightarrow
{\mathbb R}_{0}^{+}$ as follows: if $x = (x_1 , x_2 , \ldots ,$ $x_n ) \in V$,
put $\| x \|=\sqrt{x_1^2 + x_2^2 + \ldots + x_n^2}$. Then $\| \ \|$ is a
norm on $V$, called the Euclidean norm on $V$.
\end{proposition}

\begin{proof}
It is clear that $\| x \| = 0$ if and only if $x=0$ and for all
$\alpha \in {\mathbb R}_{0}^{+}$ and $x \in V$,
$\| \alpha x \| = |\alpha | \| x \|$. To show the triangle inequality it is
sufficient to apply the
Cauchy-Schwarz inequality in ${\mathbb R}_{0}^{+}$: if
$x = (x_1 , x_2 , \ldots , x_n )$ and $y = (y_1 , y_2 , \ldots , y_n )$ are
semi-vectors in $V$ then $\displaystyle\sum_{i=1}^{n} x_i y_i \leq
{\left(\displaystyle\sum_{i=1}^{n} x_i^2 \right)}^{1/2} \cdot
{\left(\displaystyle\sum_{i=1}^{n} y_i^2 \right)}^{1/2}$.
\end{proof}

In the next result we show that the Euclidean norm on $[{\mathbb R}_{0}^{+}]^{n}$
generates the Euclidean metric on it.

\begin{proposition}\label{R+1a}
Let $x = (x_1 , x_2 , \ldots ,x_n )$, $y = (y_1 , y_2 , \ldots , y_n )$ be
semi-vectors in $V = [{\mathbb R}_{0}^{+}]^{n}$. Define the function
$d:V \times V \longrightarrow {\mathbb R}_{0}^{+}$ as follows: for every fixed $i$,
if $x_i = y_i$ put $c_i =0$; if $x_i \neq y_i$, put
${\varphi}_i = {\psi}_i + c_i$, where ${\varphi}_i =\max \{x_i, y_i \}$ and
${\psi}_i =\min \{ x_i , y_i\}$ (in this case, $c_i > 0$); then consider
$d(x, y) = \sqrt{c_1^2 + \ldots + c_n^2}$. The function $d$ is a metric on $V$.
\end{proposition}

\begin{remark}
Note that in Proposition~\ref{R+1a} we could have defined $c_i$
simply by the nonnegative real number satisfying $\max
\{x_i, y_i \}=\min \{x_i, y_i \} + c_i$. However, we prefer to separate
the cases when $c_i=0$ and $c_i > 0$ in order to improve the
readability of this paper.
\end{remark}

\begin{proof}
It is easy to see that $d(x, y)=0$ if and only if $x=y$ and $d(x, y)=d(y,x)$.

We will next prove the triangle inequality. To do this, let
$x = (x_1 , x_2 , \ldots ,x_n )$, $y = (y_1 , y_2 , \ldots , y_n )$
and $z = (z_1 , z_2 , \ldots , z_n )$ be semi-vectors in $V = [{\mathbb R}_{0}^{+}]^{n}$.
We look first at a fixed $i$. If
$x_i = y_i = z_i$ or if two of them are equal then $d(x_i , z_i ) \leq d(x_i , y_i ) +
d(y_i, z_i )$. Let us then assume that $x_i$, $y_i$ and $z_i$ are pairwise distinct.
We have to analyze the six cases: $\operatorname{(1)}$ $x_i < y_i < z_i$;
$\operatorname{(2)}$ $x_i <  z_i < y_i$; $\operatorname{(3)}$ $y_i < x_i < z_i$;
$\operatorname{(4)}$ $y_i < z_i < x_i$; $\operatorname{(5)}$ $z_i
< x_i < y_i$; $\operatorname{(6)}$ $z_i < y_i < x_i$.
In order to verify the triangle inequality we will see what occurs in the worst cases.
More precisely, we assume that for all $i=1, 2, \ldots , n$ we have
$x_i < y_i < z_i$ or, equivalently, $z_i < y_i < x_i$. Since both cases are analogous we only
verify the (first) case $x_i < y_i < z_i$, for all $i$. In such cases there exist
positive real numbers $a_i$, $b_i$, for all $i=1, 2, \ldots , n$,
such that $y_i = x_i + a_i$ and $z_i = y_i + b_i$, which implies
$z_i = x_i + a_i + b_i$. We need to show that
$d(x, z) \leq d(x, y) + d(y, z)$, i.e.,
${\left(\displaystyle\sum_{i=1}^{n}(a_i + b_i)^2\right)}^{1/2} \leq
{\left(\displaystyle\sum_{i=1}^{n} a_i^2\right)}^{1/2} +
{\left(\displaystyle\sum_{i=1}^{n} b_i^2\right)}^{1/2}$.
The last inequality is equivalent to the inequality
$\displaystyle\sum_{i=1}^{n} (a_i + b_i)^2 \leq
\displaystyle\sum_{i=1}^{n} a_i^2 + \displaystyle\sum_{i=1}^{n} b_i^2 +
2{\left(\displaystyle\sum_{i=1}^{n} a_i^2 \right)}^{1/2}
\cdot {\left(\displaystyle\sum_{i=1}^{n}
b_i^2\right)}^{1/2}$. Again, the
last inequality is equivalent to $\displaystyle\sum_{i=1}^{n} a_i b_i \leq
{\left(\displaystyle\sum_{i=1}^{n} a_i^2\right)}^{1/2}\cdot {\left(\displaystyle\sum_{i=1}^{n}
b_i^2\right)}^{1/2}$, which is the Cauchy-Schwarz inequality in
${\mathbb R}_{0}^{+}$. Therefore, $d$ satisfies the triangle inequality, hence it is a
metric on $V$.
\end{proof}

\begin{remark}
Note that Proposition~\ref{R+1a} means that the Euclidean
norm on $[{\mathbb R}_{0}^{+}]^{n}$
(see Proposition~\ref{R+1}) generates the Euclidean metric on $[{\mathbb R}_{0}^{+}]^{n}$.
This result is analogous to the fact that every norm defined on vector spaces
generates a metric on it. Further, a semi-vector space $V$ is Banach
(see Definition~\ref{semiBanach}) if the norm generates a metric under which
every Cauchy sequence in $V$ converges to an element of $V$.
\end{remark}

\begin{proposition}\label{R+1b}
Let $V = [{\mathbb R}_{0}^{+}]^{n}$ and define the function
$\langle \ , \ \rangle:V\times V \longrightarrow
{\mathbb R}_{0}^{+}$ as follows: if $u = (x_1 , x_2 , \ldots , x_n )$ and
$v = (y_1 , y_2 , \ldots , y_n )$ are semi-vectors in $V$, put
$\langle u , v \rangle = \displaystyle\sum_{i=1}^{n}x_i y_i$. Then
$\langle \ , \ \rangle$ is an inner product on $V$,
called dot product.
\end{proposition}
\begin{proof}
The proof is immediate.
\end{proof}

\begin{proposition}\label{R+1c}
The dot product on $V = [{\mathbb R}_{0}^{+}]^{n}$ generates the
Euclidean norm on $V$.
\end{proposition}
\begin{proof}
If $x= (x_1 , x_2 , \ldots , x_n ) \in V$, define the norm of $x$ by
$\| x \|=\sqrt{\langle x, x\rangle}$. Note that the norm is exactly
the Euclidean norm given in Proposition~\ref{R+1}.
\end{proof}

\begin{remark}
We observe that if an inner product on a semi-vector space $V$ generates
a norm $\| \ \|$ and such a norm generates a metric $d$ on $V$,
then $V$ is a Hilbert space (according to Definition~\ref{semiHilbert}) if every Cauchy
sequence in $V$ converges w.r.t. $d$ to an element of $V$.
\end{remark}

\begin{proposition}\label{R+2}
Let $V = [{\mathbb R}_{0}^{+}]^{n}$ and define the function ${\| \ \|}_1:V
\longrightarrow {\mathbb R}_{0}^{+}$ as follows: if $x = (x_1 ,
x_2 , \ldots ,$ $x_n ) \in V$,  ${\| x \|}_1=\displaystyle\sum_{i=1}^{n} x_i$.
Then ${\| x \|}_1$ is a norm on $V$.
\end{proposition}
\begin{proof}
The proof is direct.
\end{proof}

\begin{proposition}\label{R+2a}
Let $x = (x_1 , x_2 , \ldots ,x_n )$, $y = (y_1 , y_2 , \ldots , y_n )$ be
semi-vectors in $V = [{\mathbb R}_{0}^{+}]^{n}$. Define the function
$d_1:V \times V \longrightarrow {\mathbb R}_{0}^{+}$ in the following way. For every fixed $i$,
if $x_i = y_i$, put $c_i =0$; if $x_i \neq y_i$, put
${\varphi}_i = {\psi}_i + c_i$, where ${\varphi}_i =\max \{x_i, y_i \}$ and
${\psi}_i =\min \{ x_i , y_i\}$. Let us consider that
$d_1 (x, y) = \displaystyle\sum_{i=1}^{n} c_i $.
Then the function $d_1$ is a metric on $V$ derived from the norm
${\| \ \|}_1$ shown in Proposition~\ref{R+2}.
\end{proposition}
\begin{proof}
We only prove the triangle inequality. To avoid stress of notation,
we consider the same that was considered in the proof of
Proposition~\ref{R+1a}. We then fix $i$ and only investigate the worst case
$x_i < y_i < z_i$. In this case, there exist positive real numbers
$a_i$, $b_i$ for all $i=1, 2 , \ldots , n$, such that $y_i = x_i + a_i$
and $z_i = y_i + b_i$, which implies $z_i = x_i + a_i + b_i$. Then, for all
$i$, $d_1 (x_i , z_i) \leq d_1 (x_i , y_i ) + d_1 (y_i , z_i)$; hence,
$d_1 (x, z)=\displaystyle\sum_{i=1}^{n} d_1 (x_i , z_i) =
\displaystyle\sum_{i=1}^{n} (a_i + b_i ) =
\displaystyle\sum_{i=1}^{n} a_i + \displaystyle\sum_{i=1}^{n} b_i =
\displaystyle\sum_{i=1}^{n} d_1 (x_i , y_i ) + \displaystyle\sum_{i=1}^{n}
d_1 (y_i , z_i )= d_1 (x, y) + d_1 (y, z)$. Therefore, $d_1$ is a metric on $V$.
\end{proof}

\begin{proposition}\label{R+3}
Let $V = [{\mathbb R}_{0}^{+}]^{n}$ be the Euclidean semi-vector space of
semi-dimension $n$. Define the function ${\| \ \|}_2:V \longrightarrow
{\mathbb R}_{0}^{+}$ as follows: if $x = (x_1 , x_2 , \ldots ,$ $x_n ) \in V$,
take ${\| x \|}_2=\displaystyle\max_{i} \{ x_i \}$. Then ${\| x \|}_2$
is a norm on $V$.
\end{proposition}

\begin{proposition}\label{R+3a}
Keeping the notation of Proposition~\ref{R+1a}, define the function
$d_2:V \times V \longrightarrow {\mathbb R}_{0}^{+}$ such that
$d_2 (x, y) = \max_{i} \{ c_i \}$. Then $d_2$ is a metric on $V$. Moreover, $d_2$
is obtained from the norm ${\| \ \|}_2$ exhibited in Proposition~\ref{R+3}.
\end{proposition}

\begin{proposition}\label{R+4}
The norms $\| \ \|$, ${\| \ \|}_1$ and ${\| \ \|}_2$ shown in
Propositions~\ref{R+1},~\ref{R+2} and \ref{R+3} are equivalent.
\end{proposition}
\begin{proof}
It is immediate to see that ${\| \ \|}_2 \leq \| \ \| \leq
{\| \ \|}_1 \leq n {\| \ \|}_2$.
\end{proof}

In a natural way we can define the norm of a bounded semi-linear
transformation.

\begin{definition}\label{semibounded}
Let $V$ and $W$ be two normed semi-vector spaces and let $T:V
\longrightarrow W$ be a semi-linear transformation.
We say that $T$ is bounded if there exists a real number $c > 0$
such that $\| T(v)\|\leq c \| v \|$.
\end{definition}

If $T:V \longrightarrow W$ is bounded and $v \neq 0$
we can consider the quotient $\frac{\| T(v)\|}{\| v \|}$. Since
such a quotient is upper bounded by $c$, the supremum $\displaystyle
\sup_{v \in V, v\neq 0}\frac{\| T(v)\|}{\| v \|}$
exists and it is at most $c$. We then define
$$\| T \|= \displaystyle\sup_{v \in V, v\neq 0}\frac{\| T(v)\|}{\| v \|}.$$

\begin{proposition}\label{R+5}
Let $T: V \longrightarrow W$ be a bounded semi-linear transformation.
Then the following hold:
\begin{itemize}
\item [ $\operatorname{(1)}$] $T$ sends bounded sets in bounded sets;
\item [ $\operatorname{(2)}$] $\| T \|$ is a norm, called norm of $T$;
\item [ $\operatorname{(3)}$] $\| T \|$ can be written in the form
$\| T \|= \displaystyle\sup_{v \in V, \| v \| = 1 } \| T(v) \|$.
\end{itemize}
\end{proposition}
\begin{proof}
Items~$\operatorname{(1)}$~and~ $\operatorname{(2)}$ are immediate.
The proof of Item~$\operatorname{(3)}$ is analogous to the
standard proof but we present it here to guarantee that our mathematical
tools are sufficient to perform it. Let $v\neq 0$ be a semi-vector
with norm $\| v \|= a \neq 0$ and set $u=(1/a)v$. Thus,
$\| u \| =1$ and since $T$ is semi-linear one has
$$\| T \|= \displaystyle\sup_{v \in V, v\neq 0} \frac{1}{a}\|
T(v)\|=\displaystyle\sup_{v \in V, v\neq 0} \| T( (1/a) v) \|=
\displaystyle\sup_{u \in V, \| u \| =1} \| T(u)\|=$$ $=
\displaystyle\sup_{v \in V, \| v \| =1} \| T(v)\|$.
\end{proof}


\subsubsection{The Semi-Spaces ${l}_{+}^{\infty}$, ${l}_{+}^{p}$
and ${\operatorname{C}}_{+}[a, b]$}\label{subsubsec1}

In this subsection we investigate topological aspects of
some semi-vector spaces over ${\mathbb R}_{0}^{+}$ such
as completeness and separability. We investigate the sequence spaces
${l}_{+}^{\infty}$, ${l}_{+}^{p}$, ${\operatorname{C}}_{+}[a, b]$,
which will be defined in the sequence.

We first study the space ${l}_{+}^{\infty}$, the set
of all bounded sequences of nonnegative real numbers.
Before studying such a space we must define a metric on it,
since the metric in $l^{\infty}$
which is defined as $ d(x, y)=\displaystyle\sup_{i \in
{\mathbb N}} | {x}_i - {y}_i |$,
where $x = ({x}_i )$ and $y = ({y}_i )$ are sequences in
$l^{\infty}$,
has no meaning to us, because there is no sense in
considering $- {y}_i$ if ${y}_i > 0$. Based on this fact,
we circumvent
this problem by utilizing the total order of ${\mathbb R}$
according to Proposition~\ref{R+1a}. Let $x = ({\mu}_i )$ and
$y = ({\nu}_i )$ be sequences in $l_{+}^{\infty}$. We then fix $i$,
and define $c_i$ as was done in Proposition~\ref{R+1a}:
if ${\mu}_i = {\nu}_i $ then we put $c_i = 0$; if ${\mu}_i \neq {\nu}_i $,
let ${\gamma}_i=\max \{{\mu}_i , {\nu}_i \}$
and ${\psi}_i= \min \{{\mu}_i , {\nu}_i \}$; then there exists a positive
real number $c_i$ such that ${\gamma}_i = {\psi}_i + c_i$
and, in place of $| {\mu}_i - {\nu}_i |$, we put $c_i$. Thus,
our metric becomes
\begin{eqnarray}\label{lmetric}
d(x, y) = \displaystyle\sup_{i \in {\mathbb N}} \{c_i \}.
\end{eqnarray}

It is clear that $d(x, y)$ shown in Eq.~(\ref{lmetric}) defines a metric. However,
we must show that the tools that we have are sufficient to proof this fact,
once we are working on ${\mathbb R}_{0}^{+}$.

\begin{proposition}\label{metricsup}
The function $d$ shown in Eq.~(\ref{lmetric}) is a metric on ${l}_{+}^{\infty}$.
\end{proposition}
\begin{proof}
It is clear that $d(x,y)\geq 0$ and $d(x,y)= 0 \Longleftrightarrow
x=y$. Let $x = ({\mu}_i )$ and $y = ({\nu}_i )$ be two sequences
in $l_{+}^{\infty}$. Then, for every fixed $i \in {\mathbb N}$, if $c_i=
d({\mu}_i , {\nu}_i )=0$ then ${\mu}_i = {\nu}_i$, i.e., $d({\mu}_i
, {\nu}_i )=d({\nu}_i , {\mu}_i )$. If $c_i > 0$
then $c_i= d({\mu}_i , {\nu}_i )$ is computed by ${\gamma}_i =
{\psi}_i + c_i$, where ${\gamma}_i=\max \{{\mu}_i , {\nu}_i \}$
and ${\psi}_i= \min \{{\mu}_i , {\nu}_i \}$. Hence,
$d({\nu}_i , {\mu}_i ) = c_i^{*}$ is computed by
${\gamma}_i^{*} = {\psi}_i^{*} + c_i^{*}$, where
${\gamma}_i^{*}=\max \{{\nu}_i, {\mu}_i \}$ and ${\psi}_i^{*}=
\min \{{\nu}_i, {\mu}_i \}$, which implies $d({\mu}_i , {\nu}_i )
=d({\nu}_i , {\mu}_i )$. Taking the supremum over all $i$'s we have
$d(x, y) = \displaystyle\sup_{i \in {\mathbb N}} \{c_i \}=
\displaystyle\sup_{i \in {\mathbb N}} \{c_i^{*} \}=d(y, x)$.

To show the triangle inequality, let $x = ({\mu}_i )$,
$y = ({\nu}_i )$ and $z=({\eta}_i)$ be sequences in $l_{+}^{\infty}$.
For every fixed $i$, we will prove that
$d({\mu}_i , {\eta}_i )\leq d({\mu}_i , {\nu}_i ) + d({\nu}_i ,
{\eta}_i )$. If ${\nu}_i = {\mu}_i = {\eta}_i$, the result is
trivial. If two of them are equal, the result is also trivial. Assume
that ${\mu}_i$, ${\nu}_i$ and ${\eta}_i$ are pairwise distinct.
As in the proof of Proposition~\ref{R+1a}, we must investigate the six cases:\\
$\operatorname{(1)}$ ${\mu}_i < {\nu}_i < {\eta}_i$;
$\operatorname{(2)}$ ${\mu}_i <  {\eta}_i < {\nu}_i$;
$\operatorname{(3)}$ ${\nu}_i < {\mu}_i < {\eta}_i$;
$\operatorname{(4)}$ ${\nu}_i < {\eta}_i < {\mu}_i$;
$\operatorname{(5)}$ ${\eta}_i < {\mu}_i < {\nu}_i$;
$\operatorname{(6)}$ ${\eta}_i < {\nu}_i < {\mu}_i$.
We only show $\operatorname{(1)}$ and $\operatorname{(2)}$.

To show $\operatorname{(1)}$, note that there exist positive real
numbers $c_i$ and $c_i^{'}$ such that ${\nu}_i =
{\mu}_i + c_i$ and ${\eta}_i = {\nu}_i + c_i^{'}$, which implies $\eta_i =
\mu_i + c_i + c_i^{'}$. Hence, $d({\mu}_i , {\eta}_i )=c_i + c_i^{'}=
d({\mu}_i , {\nu}_i ) + d({\nu}_i , {\eta}_i )$.

Let us show $\operatorname{(2)}$. There exist positive real
numbers $b_i$ and $b_i^{'}$ such that  ${\eta}_i =
{\mu}_i + b_i$ and ${\nu}_i={\eta}_i + b_i^{'}$, so ${\nu}_i = {\mu}_i
+ b_i + b_i^{'}$. Therefore, $d({\mu}_i , {\eta}_i )=b_i < d({\mu}_i ,
{\nu}_i ) + d({\nu}_i , {\eta}_i )=b_i + 2b_i^{'}$.

Taking the supremum over all $i$'s we have
$\displaystyle\sup_{i \in {\mathbb N}} \{d({\mu}_i ,
{\eta}_i ) \} \leq \displaystyle\sup_{i \in {\mathbb N}} \{d({\mu}_i ,
{\nu}_i )\} + \displaystyle\sup_{i \in {\mathbb N}} \{d({\nu}_i , {\eta}_i ) \}$, i.e.,
$d(x, z) \leq d(x, y) + d(y, z)$. Therefore, $d$ is a metric on ${l}_{+}^{\infty}$.
\end{proof}

\begin{definition}\label{defl}
The metric space ${l}_{+}^{\infty}$ is the set of all bounded
sequences of nonnegative real numbers equipped with the metric
$d(x, y) = \displaystyle\sup_{i \in {\mathbb N}} \{c_i \}$ given previously.
\end{definition}

We prove that ${l}_{+}^{\infty}$ equipped with the previous metric is complete.

\begin{theorem}\label{lcomplete}
The space ${l}_{+}^{\infty}$ with the metric $d(x, y) = \displaystyle\sup_{i
\in {\mathbb N}} \{c_i \}$ shown above is complete.
\end{theorem}
\begin{proof}
The proof follows the same line as the standard proof of completeness
of ${l}^{\infty}$; however it is necessary to adapt it
to the metric (written above) in terms of nonnegative real numbers. Let $(x_n)$
be a Cauchy sequence in ${l}_{+}^{\infty}$, where $x_i =
({\eta}_{1}^{(i)}, {\eta}_{2}^{(i)}, \ldots )$. We must show that
$(x_n )$ converges to an element of ${l}_{+}^{\infty}$. As $(x_n)$
is Cauchy, given $\epsilon > 0$, there exists a positive integer
$K$ such that, for all $n, m > K$, $$d(x_n, x_m)=\displaystyle\sup_{j
\in {\mathbb N}} \{c_j^{(n, m)} \} < \epsilon,$$
where $c_j^{(n, m)}$ is a nonnegative real number such that, if
${\eta}_{j}^{(n)}={\eta}_{j}^{(m)}$ then $c_j^{(n, m)}=0$, and
if ${\eta}_{j}^{(n)} \neq {\eta}_{j}^{(m)}$ then $c_j^{(n, m)}$
is given by $\max \{{\eta}_{j}^{(n)}, {\eta}_{j}^{(m)}\} = \min
\{{\eta}_{j}^{(n)}, {\eta}_{j}^{(m)}\} +c_j^{(n, m)}$. This
implies that for each fixed $j$ one has
\begin{eqnarray}\label{distCauchy1}
c_j^{(n, m)} < \epsilon,
\end{eqnarray}
where $n, m > K$. Thus, for each fixed $j$, it follows that
$({\eta}_{j}^{(1)}, {\eta}_{j}^{(2)}, \ldots )$ is a
Cauchy sequence in ${\mathbb R}_{0}^{+}$. Since ${\mathbb R}_{0}^{+}$ is a
complete metric space, the sequence $({\eta}_{j}^{(1)}, {\eta}_{j}^{(2)},
\ldots )$ converges to an element ${\eta}_{j}$ in ${\mathbb R}_{0}^{+}$.
Hence, for each $j$, we form the sequence $x$ whose coordinates are
the limits ${\eta}_{j}$, i.e., $x =({\eta}_{1}, {\eta}_{2}, {\eta}_{3},
\ldots )$. We must show that $x \in {l}_{+}^{\infty}$ and $x_n
\longrightarrow x$.

To show that $x$ is a bounded sequence, let us consider the number
$c_j^{(n, \infty)}$ defined as follows: if ${\eta}_{j} =
{\eta}_{j}^{(n)}$ then $c_j^{(n, \infty)}=0$, and if ${\eta}_{j} \neq {\eta}_{j}^{(n)}$,
define $c_j^{(n, \infty)}$ be the positive real number satisfying
$\max \{{\eta}_{j} , {\eta}_{j}^{(n)} \}= \min \{{\eta}_{j} , {\eta}_{j}^{(n)} \}
+ c_j^{(n, \infty)}$. From the inequality $(\ref{distCauchy1})$ one has

\begin{eqnarray}\label{distCauchy2}
c_j^{(n, \infty)}\leq\epsilon .
\end{eqnarray}
Because ${\eta}_{j} \leq {\eta}_{j}^{(n)} + c_j^{(n, \infty)}$ and since
${\eta}_{j}^{(n)} \in l_{+}^{\infty}$, it follows that ${\eta}_{j}$ is
a bounded sequence for every $j$. Hence,  $x = ({\eta}_{1},
{\eta}_{2}, {\eta}_{3}, \ldots ) \in {l}_{+}^{\infty}$.
From $(\ref{distCauchy2})$ we have
$$\displaystyle\sup_{j \in {\mathbb N}} \{c_j^{(n, \infty)} \} \leq \epsilon,$$
which implies that $x_n \longrightarrow x$. Therefore, $l_{+}^{\infty}$ is complete.
\end{proof}

Although $l_{+}^{\infty}$ is a complete metric space, it is not separable.

\begin{theorem}\label{lnotsep}
The space ${l}_{+}^{\infty}$ with the metric $d(x, y) = \displaystyle\sup_{i
\in {\mathbb N}} \{c_i \}$ is not separable.
\end{theorem}
\begin{proof}
The proof is the same as shown in \cite[1.3-9]{Kreyszig:1978}, so it is omitted.
\end{proof}

Let us define the space analogous to the space $l^p$.

\begin{definition}\label{deflp}
Let $p \geq 1$ be a fixed real number. The set ${l}_{+}^{p}$ consists
of all sequences $x =({\eta}_{1}, {\eta}_{2}, {\eta}_{3}, \ldots )$
of nonnegative real numbers such that $\displaystyle\sum_{i=1}^{\infty} ({\eta}_{i})^{p} < \infty$,
whose metric is defined by
$ d(x, y)={\left[\displaystyle\sum_{i=1}^{\infty} {[c_{i}]}^{p}\right]}^{1/p}$,
where $y =({\mu}_{1}, {\mu}_{2}, {\mu}_{3}, \ldots )$ and $c_i$ is
defined as follows: $c_i = 0$ if ${\mu}_i = {\eta}_i $, and if ${\mu}_i > {\eta}_i$
(respect. ${\eta}_i > {\mu}_i$) then $c_i > 0$ is such that ${\mu}_i = {\eta}_i + c_i$.
\end{definition}

\begin{theorem}\label{lp+complete}
The space ${l}_{+}^{p}$ with the metric $ d(x,y)=
{\left[\displaystyle\sum_{i=1}^{\infty}
{[c_{i}]}^{p}\right]}^{1/p}$ exhibited above is complete.
\end{theorem}
\begin{proof}
Recall that given two sequences $({\mu}_i)$ and $({\eta}_i )$
in ${l}_{+}^{p}$ the Minkowski inequality for sums reads as
\begin{eqnarray*}
{\left[\displaystyle\sum_{i=1}^{\infty} {|{\mu}_i +
{\eta}_i |}^{p}\right]}^{1/p} \leq {\left[\displaystyle
\sum_{j=1}^{\infty} {|{\mu}_j|}^{p}\right]}^{1/p} + {\left[\displaystyle
\sum_{k=1}^{\infty} {|{\eta}_k|}^{p}\right]}^{1/p}.
\end{eqnarray*}
Applying the Minkowski inequality as per \cite[1.5-4]{Kreyszig:1978}
with some adaptations, it follows that $d(x,y)$ is, in fact,
a metric. In order to prove the completeness of ${l}_{+}^{p}$, we proceed
similarly as in the proof of Theorem~\ref{lcomplete} with some
adaptations. The main adaptation is performed according to
the proof of completeness of $l^p$ in \cite[1.5-4]{Kreyszig:1978}
replacing the last equality $x=x_m +( x - x_m) \in l^p$
(after Eq.~(5)) by two equalities in order to avoid negative real numbers.
\begin{enumerate}
\item [ $\operatorname{(1)}$] If the $i$-th coordinate
$x^{(i)}- x_{m}^{(i)}$ of the sequence $x- x_m$ is
positive, then define $c_{m}^{(i)} = x^{(i)}- x_{m}^{(i)}$ and write
$x^{(i)} = x_{m}^{(i)} + c_{m}^{(i)}$. From Minkowski
inequality, it follows that the sequence $(x^{(i)})_i$
is in $l_{+}^{p}$.
\item [ $\operatorname{(2)}$] If $x^{(j)}- x_{m}^{(j)}$
is negative, then define $c_{m}^{(j)}= x_{m}^{(j)} -
x^{(j)}$ and write $x_{m}^{(j)}= x^{(j)} + c_{m}^{(j)} $. Since
$x_m \in l_{+}^{p}$, from the comparison criterion for
positive series it follows that the sequence $(x^{(j)})_j$ is also in $l_{+}^{p}$.
\end{enumerate}
\end{proof}

\begin{theorem}\label{lp+separable}
The space ${l}_{+}^{p}$ is separable.
\end{theorem}
\begin{proof}
The proof follows the same line of \cite[1.3-10]{Kreyszig:1978}.
\end{proof}

\begin{definition}\label{continon[a,b]}
Let $I=[a, b]$ be a closed interval in ${\mathbb R}_{0}^{+}$,
where $a\geq 0$ and $a < b$. Then ${\operatorname{C}}_{+}[a, b]$ is
the set of all continuous nonnegative real valued functions on $I=[a, b]$,
whose metric is defined by $d(f(t), g(t)) =
\displaystyle\max_{t \in I} \{c(t)\}$, where $c(t)$ is given by
$\max \{ f(t), g(t) \} =\min \{ f(t), g(t) \} + c(t)$.
\end{definition}

\begin{theorem}\label{cont[a,b]complete}
The metric space $({\operatorname{C}}_{+}[a, b], d)$, where $d$ is given in
Definition~\ref{continon[a,b]}, is complete.
\end{theorem}
\begin{proof}
The proof follows the same lines as the standard one with some modifications.
Let $(f_{m})$ be a Cauchy sequence in ${\operatorname{C}}_{+}[a, b]$. Given $\epsilon > 0$
there exists a positive integer $N$ such that, for all $m, n > N$, it follows that
\begin{eqnarray}\label{In1}
d(f_{m} , f_{n}) = \displaystyle\max_{t \in I} \{c_{m, n} (t)\} < \epsilon,
\end{eqnarray}
where $\max \{ f_{m} (t) , f_{n} (t) \} = \min \{ f_{m} (t) , f_{n} (t) \} + c_{m, n}(t)$.
Thus, for any fixed $t_0 \in I$ we have $c_{m, n} (t_0 )  < \epsilon$,
for all $m, n > N$. This means that $(f_1 (t_0 ), f_2 (t_0 ), \ldots )$ is a
Cauchy sequence in ${\mathbb R}_{0}^{+}$, which converges to $f(t_0 )$ when
$m \longrightarrow \infty$ since ${\mathbb R}_{0}^{+}$ is complete. We then
define a function $f: [a, b] \longrightarrow {\mathbb R}_{0}^{+}$ such that
for each $t \in [a, b]$, we put $f(t)$.
Taking $n \longrightarrow \infty$ in (\ref{In1}) we obtain
$\displaystyle\max_{t \in I} \{c_{m} (t)\} \leq \epsilon$ for all $m > N$, where
$\max \{ f_{m} (t) , f(t) \} = \min \{ f_{m} (t) , f(t) \} + c_{m}(t)$, which
implies $c_{m}(t)\leq \epsilon$ for all $t \in I$. This fact means that
$(f_{m}(t))$ converges to $f(t)$ uniformly on $I$, i.e., $f \in
{\operatorname{C}}_{+}[a, b]$ because the functions $f_{m}$'s are continuous on $I$.
Therefore, ${\operatorname{C}}_{+}[a, b]$ is complete, as desired.
\end{proof}

\subsection{Interesting Semi-Vector Spaces}\label{subsec2}

In this section we exhibit semi-vector spaces over $K= {\mathbb R}_{0}^{+}$
derived from semi-metrics, semi-metric-preserving functions, semi-norms,
semi-inner products and sub-linear functionals.

\begin{theorem}\label{teo1}
Let $X$ be a semi-metric space and ${ \mathcal M}_{X}=\{ d: X \times X\longrightarrow
{\mathbb R}; d$ $\operatorname{is \ a \ semi-metric \ on} X\}$.
Then $({ \mathcal M}_{X}, +, \cdot )$ is a semi-vector space over ${\mathbb R}_{0}^{+}$,
where $+$ and $\cdot$ are the
addition and the scalar multiplication (in ${\mathbb R}_{0}^{+}$) pointwise,
respectively.
\end{theorem}
\begin{proof}
We first show that ${ \mathcal M}_{X}$ is closed under addition.
Let $d_1 , d_2 \in { \mathcal M}_{X}$ and set
$d:= d_1 + d_2$. It is clear that $d$ is nonnegative real-valued
function. Moreover, for all $x, y \in X$, $d(x, y) = d(y, x)$.
Let $x \in X$; $d(x, x) = d_1(x, x) + d_2 (x,x) =0$.
For all $x, y, z \in X$, $d(x, z)=d_1 (x, z) + d_2 (x, z)\leq [d_1 (x, y) + d_2 (x, y)]+
[d_1 (y, z) + d_2 (y, z)]= d(x, y) + d(y, z)$.

Let us show that ${ \mathcal M}_{X}$ is closed under scalar multiplication. Let $d_1
\in { \mathcal M}_{X}$ and define $d = \lambda d_1$, where
$\lambda \in {\mathbb R}_{0}^{+}$. It is clear that $d$ is real-valued nonnegative and for all
$x, y \in X$, $d(x, y)=d(y, x)$. Moreover, if $x \in X$, $d(x, x)=0$.
For all $x, y, z \in X$, $d(x, z)=\lambda d_1 (x, z)\leq \lambda [d_1 (x, y)
+ d_1 (y, z)]= d(x, y) + d(y, z)$.
This means that ${ \mathcal M}_{X}$ is closed under scalar multiplication.

It is easy to see that $({ \mathcal M}_{X}, +, \cdot )$ satisfies the
other conditions of Definition~\ref{defSVS}.
\end{proof}

Let $(X, d)$ be a metric space. In~\cite{Corazza:1999}, Corazza investigated
interesting functions $f:{\mathbb R}_{0}^{+}\longrightarrow {\mathbb R}_{0}^{+}$
such that the composite of $f$ with $d$, i.e., $X \times X \xrightarrow{d}
{{\mathbb R}_{0}^{+}} \xrightarrow{f} {{\mathbb R}_{0}^{+}}$ also generates
a metric on $X$. Let us put this concept formally.

\begin{definition}\label{metricprese}
Let $f:{\mathbb R}_{0}^{+}\longrightarrow {\mathbb R}_{0}^{+}$
be a function. We say that $f$ is metric-preserving if for
all metric spaces $(X, d)$, the composite $f \circ d$ is a metric.
\end{definition}

To our purpose we will consider semi-metric preserving functions as follows.

\begin{definition}\label{semi-metricprese}
Let $f:{\mathbb R}_{0}^{+}\longrightarrow {\mathbb R}_{0}^{+}$ be a
function. We say that $f$ is semi-metric-preserving if for
all semi-metric spaces $(X, d)$, the composite $f \circ d$ is a semi-metric.
\end{definition}

We next show that the set of semi-metric preserving functions has a semi-vector
space structure.

\begin{theorem}\label{teo1a}
Let ${ \mathcal F}_{pres}=\{ f:{\mathbb R}_{0}^{+}\longrightarrow
{\mathbb R}_{0}^{+}; f \operatorname{is \ semi-metric \ preserving} \}$.
Then $({ \mathcal F}_{pres}, +, \cdot )$ is a semi-vector space over ${\mathbb R}_{0}^{+}$,
where $+$ and $\cdot$ are the addition and the scalar multiplication
(in ${\mathbb R}_{0}^{+}$) pointwise, respectively.
\end{theorem}
\begin{proof}
We begin by showing that ${ \mathcal F}_{pres}$ is closed under addition
and scalar multiplication pointwise.

Let $f, g \in { \mathcal F}_{pres}$. Given a semi-metric space $(X, d)$, we must prove that
$(f + g)\circ d$ is also semi-metric preserving. We know
that $[(f + g)\circ d] (x, y ) \geq 0$ for all $x, y \in X$. Let $x \in X$; then
$[(f + g)\circ d ](x, x )= f(d(x, x)) + g (d(x, x)) = 0$. It is clear that
$[(f + g ) \circ d](x, y)= [(f + g ) \circ d](y, x)$. Let $x, y, z \in X$. One has:
$[(f + g ) \circ d](x, y)= f(d(x, y)) + g(d(x, y))\leq [f(d(x, z))+ g(d(x, z))]+
[f(d(z, y))+ g(d(z, y))]= (f + g)(d(x, z)) + (f + g)(d(z, y))=
[(f + g)\circ d](x, z) + [(f + g)\circ d](z, y) $.

Here, we show that for each $f \in { \mathcal F}_{pres}$ and $ \alpha \in
{\mathbb R}_{0}^{+}$, it follows that $ \alpha f \in { \mathcal F}_{pres}$.
We show only the triangular inequality since the other conditions are immediate.
Let us calculate: $[\alpha f \circ d](x, y)= \alpha f (d(x, y))\leq
\alpha f (d(x, z)) + \alpha f (d(z, y)) = [\alpha f \circ d](x, z) + [\alpha f \circ d](z, y)$.

The null vector is the null function $0_{f}:{\mathbb R}_{0}^{+}\longrightarrow
{\mathbb R}_{0}^{+}$. The other conditions are easy to verify.
\end{proof}

\begin{theorem}\label{teo2}
Let $V$ be a semi-normed real vector space and ${ \mathcal N}_{V}=
\{ \| \ \|: V\longrightarrow {\mathbb R}; \| \ \|$
$\operatorname{is \ a \ semi-norm \ on} V\}$. Then $({ \mathcal N}_{V}, +,
\cdot )$ is a semi-vector space over ${\mathbb R}_{0}^{+}$,
where $+$ and $\cdot$ are addition and scalar multiplication
(in ${\mathbb R}_{0}^{+}$) pointwise, respectively.
\end{theorem}
\begin{proof}
From hypotheses, ${ \mathcal N}_{V}$ is non-empty. Let ${\| \ \|}_{1} ,
{\| \ \|}_{2} \in { \mathcal N}_{V}$ and set $\| \ \|:=
{\| \ \|}_{1} + {\| \ \|}_{2}$. For all $v \in V$, $\| v \|\geq 0$.
If $v \in V$ and $\alpha \in {\mathbb R}$ then $\| \alpha v \|=|\alpha| \| v \|$.
For every $u, v \in V$, it follows that $\| u + v \|:= {\| u + v \|}_{1} +
{\| u + v \|}_{2}\leq ({ \| u \|}_{1} + {\| u \|}_{2} ) +
({\| v \|}_{1} + {\| v \|}_{2})= \| u \| + \| v \|$. Hence, ${ \mathcal N}_{V}$
is closed under addition.

We next show that ${ \mathcal N}_{V}$ is closed under scalar multiplication.
Let ${\| \ \|}_{1} \in { \mathcal N}_{V}$ and define
$\| \ \|:= \lambda {\| \ \|}_{1}$, where $\lambda \in {\mathbb R}_{0}^{+}$. For all
$v \in V$, $\| v \|\geq 0$. If $\alpha \in {\mathbb R}$ and $ v \in V$,
$ \| \alpha v \|= |\alpha |( \lambda {\| v\|}_{1})= |\alpha | \| v \|$.
Let $u, v \in V$. Then $\| u + v \|\leq \lambda {\| u \|}_{1}+
\lambda {\| v \|}_{1}=\|u\| + \|v\|$. Therefore, ${ \mathcal N}_{V}$
is closed under addition and scalar multiplication over ${\mathbb R}_{0}^{+}$.

The zero vector is the null function $ \textbf{0}: V \longrightarrow
{\mathbb R}$. The other conditions
of Definition~\ref{defSVS} are straightforward.
\end{proof}

\begin{remark}
Note that ${ \mathcal N}_{V}^{\diamond}=\{\| \ \|: V\longrightarrow {\mathbb R};
\| \ \|$ $\operatorname{is \ a \ norm \ on} V\}$
is also closed under both function addition and scalar multiplication pointwise.
\end{remark}

\begin{lemma}\label{prop1}
Let $T:V\longrightarrow W$ be a linear transformation.
\begin{itemize}
\item [ $\operatorname{(1)}$] If $\| \ \|:W\longrightarrow {\mathbb R}$ is a semi-norm on
$W$ then $\| \ \|\circ T: V \longrightarrow {\mathbb R}$ is a semi-norm on $V$.
\item [ $\operatorname{(2)}$] If $T$ is injective linear and $\| \ \|:
W\longrightarrow {\mathbb R}$ is a norm on $W$ then $\| \ \|\circ T$ is a norm on $V$.
\end{itemize}
\end{lemma}
\begin{proof}
We only show Item~$\operatorname{(1)}$. It is clear that $[\| \ \|\circ T](v)
\geq 0$ for all $v \in V$. For all $\alpha
\in {\mathbb R}$ and $v \in V$, $[\| \ \|\circ T](\alpha v)=
| \alpha |  \| T(v) \| = | \alpha | [\| \ \|\circ T](v)$. Moreover, $ \forall \ v_1 , v_2 \in V$,
$[\| \ \|\circ T](v_1 + v_2)\leq [\| \ \|\circ T](v_1 )+ [\| \ \|\circ T](v_2 )$.
Therefore, $\| \ \|\circ T$ is a semi-norm on $V$.
\end{proof}

\begin{theorem}\label{teo2a}
Let $V$ and $W$ be two semi-normed vector spaces and $T:V\longrightarrow W$ be
a linear transformation. Then
$${ \mathcal N}_{V_{T}}=\{ \| \ \| \circ T:
V\longrightarrow {\mathbb R}; \| \ \| \operatorname{is \ a \ semi-norm \ on} W\}$$
is a semi-subspace of $({ \mathcal N}_{V}, +, \cdot )$.
\end{theorem}

\begin{proof}
From hypotheses, it follows that ${ \mathcal N}_{V_{T}}$ is non-empty.
From Item~$\operatorname{(1)}$ of Lemma~\ref{prop1}, it follows that $\| \ \|\circ T$
is a semi-norm on $V$. Let $f, g \in { \mathcal N}_{V_{T}}$, i.e.,
$f = {\| \ \|}_1 \circ T$ and $g = {\| \ \|}_2 \circ T$, where ${\| \ \|}_1$ and ${\| \ \|}_2$
are semi-norms on $W$. Then $f + g = [ {\| \ \|}_1 + {\| \ \|}_2 ]\circ T \in { \mathcal N}_{V_{T}}$.
For every nonnegative real number $\lambda$ and $f \in { \mathcal N}_{V_{T}}$,
$\lambda f = \lambda [ \| \ \|\circ T] = (\lambda \| \ \| )\circ T \in { \mathcal N}_{V_{T}}$.
\end{proof}

\begin{theorem}\label{teo2b}
Let ${\mathcal N}$ be the class whose members are $\{{ \mathcal N}_{V}\}$, where the
${ \mathcal N}_{V}$ are given in Theorem~\ref{teo2}. Let
$\operatorname{Hom}({\mathcal N})$ be the class whose members are the sets
$$\operatorname{hom}({ \mathcal N}_{V}, { \mathcal N}_{W})=\{
F_T:{ \mathcal N}_{V}\longrightarrow { \mathcal N}_{W};
F_T ( {\| \ \|}_{V})= {\| \ \|}_{V} \circ T\},$$
where $T: W \longrightarrow V$ is a linear transformation and
${\| \ \|}_{V}$ is a semi-norm on $V$. Then $({\mathcal N}, \operatorname{Hom}({\mathcal N}),
Id, \circ )$ is a category.
\end{theorem}
\begin{proof}
The sets $\operatorname{hom}({ \mathcal N}_{V}, { \mathcal N}_{W})$ are pairwise disjoint.
For each ${ \mathcal N}_{V}$, there exists $Id_{({ \mathcal N}_{V})}$ given by
$Id_{({ \mathcal N}_{V})} ({\| \ \|}_{V})={\| \ \|}_{V}={\| \ \|}_{V}\circ Id_{(V)}$.
It is clear that if ${F}_{T}:{ \mathcal N}_{V}\longrightarrow { \mathcal N}_{W}$ then
${F}_{T}\circ Id_{({ \mathcal N}_{V})} = {F}_{T}$ and
$Id_{({ \mathcal N}_{W})}\circ {F}_{T} = {F}_{T}$.

It is easy to see that for every $T:W\longrightarrow V$ linear transformation,
the map $F_{T}$ is semi-linear, i.e.,
$F_{T}({\| \ \|}_{V}^{(1)} + {\| \ \|}_{V}^{(2)})=
F_{T}({\| \ \|}_{V}^{(1)}) + F_{T}({\| \ \|}_{V}^{(2)})$ and
$F_{T}(\lambda {\| \ \|}_{V})= \lambda F_{T}({\| \ \|}_{V})$,
for every ${\| \ \|}_{V}, {\| \ \|}_{V}^{(1)}, {\| \ \|}_{V}^{(2)} \in { \mathcal N}_{V}$
and $\lambda \in {\mathbb R}_{0}^{+}$.

Let ${ \mathcal N}_{U}, { \mathcal N}_{V}, { \mathcal N}_{W},
{ \mathcal N}_{X} \in {\mathcal N}$ and $F_{T_1} \in
\operatorname{hom}({ \mathcal N}_{U}, { \mathcal N}_{V})$,
$F_{T_2} \in \operatorname{hom}({ \mathcal N}_{V}, { \mathcal N}_{W})$,
$F_{T_3} \in \operatorname{hom}({ \mathcal N}_{W}, { \mathcal N}_{X})$, i.e.,
$${ \mathcal N}_{U}\xrightarrow{F_{T_1}} { \mathcal N}_{V}\xrightarrow{F_{T_2}} { \mathcal N}_{W}
\xrightarrow{F_{T_3}} { \mathcal N}_{X}.$$
The linear transformations are of the forms
$$X\xrightarrow{T_3} W\xrightarrow{T_2} V \xrightarrow{T_1} U
\xrightarrow{{\| \ \|}_{U}} {\mathbb R}.$$
The associativity $(F_{T_3}\circ F_{T_2})\circ F_{T_1}=F_{T_3}\circ (F_{T_2}\circ F_{T_1})$
follows from the associativity of composition of maps. Moreover, the map
$F_{T_3}\circ F_{T_2}\circ F_{T_1} \in \operatorname{Hom}({\mathcal N})$ because
$F_{T_3}\circ F_{T_2}\circ F_{T_1} = ({\| \ \|}_{U})\circ (T_1\circ T_2\circ T_3)$ and
$T_1\circ T_2\circ T_3$ is a linear transformation. Therefore, $({\mathcal N},
\operatorname{Hom}({\mathcal N}), Id, \circ )$ is a category, as required.
\end{proof}

\begin{theorem}\label{teo3}
Let $V$ be a real vector space endowed with a semi-inner product and let
${ \mathcal P}_{V}=\{ \langle \ ,
\ \rangle: V\times V\longrightarrow {\mathbb R}; \langle \
, \ \rangle$ $\operatorname{is \ a \ semi-inner \ product \ on} V\}$.
Then $({ \mathcal P}_{V}, +, \cdot )$ is a semi-vector space
over ${\mathbb R}_{0}^{+}$, where $+$ and $\cdot$ are
addition and scalar multiplication (in ${\mathbb R}_{0}^{+}$)
pointwise, respectively.
\end{theorem}
\begin{proof}
The proof is analogous to that of Theorems~\ref{teo1}~and~\ref{teo2}.
\end{proof}

\begin{proposition}\label{prop2}
Let $V, W$ be two vector spaces and $T_1 , T_2:V\longrightarrow W$ be two
linear transformations. Let us
consider the map $T_1 \times T_2 : V \times V \longrightarrow W\times W$ given
by $T_1 \times T_2 (u, v) = (T_1(u), T_2 (v))$. If $\langle \ , \ \rangle$
is a semi-inner product on $W$ then $\langle \ , \ \rangle \circ
T_1 \times T_2$ is a semi-inner product on $V$.
\end{proposition}
\begin{proof}
The proof is immediate, so it is omitted.
\end{proof}

Let $V$ be a real vector space. Recall that a sub-linear functional on $V$
is a functional $t: V\longrightarrow {\mathbb R}$ which is sub-additive:
$\forall \ u, v \in V$, $t(u + v)\leq t(u) + t(v)$; and positive-homogeneous:
$\forall \ \alpha \in {\mathbb R}_{0}^{+}$ and $\forall \ v \in V$,
$t(\alpha v ) =\alpha t(v)$.

\begin{theorem}\label{teo4}
Let $V$ be a real vector space. Let us consider ${ \mathcal S}_{V}=
\{ S: V\longrightarrow
{\mathbb R};$ $S \operatorname{is}  \operatorname{sub-linear} \operatorname{on} V\}$.
Then $({ \mathcal S}_{V}, +, \cdot )$ is a semi-vector space on
${\mathbb R}_{0}^{+}$, where $+$ and $\cdot$ are
addition and scalar multiplication (in ${\mathbb R}_{0}^{+}$) pointwise, respectively.
\end{theorem}
\begin{proof}
The proof follows the same line of that of Theorems~\ref{teo1}~and~\ref{teo2}~and~\ref{teo3}.
\end{proof}

\subsection{Semi-Algebras}\label{subsec4}

We start this section by recalling the definition of semi-algebra and
semi-sub-algebra. For more details the reader can consult
\cite{Gahler:1999}. In \cite{Olivier:1995}, Olivier and Serrato investigated
relation semi-algebras, i.e., a semi-algebra being both a Boolean algebra and an
involutive semi-monoid, satisfying some conditions
(see page 2 in Ref.~\cite{Olivier:1995} for more details). Roy
\cite{Roy:1970} studied the semi-algebras of continuous and
monotone functions on compact ordered spaces.

\begin{definition}\label{semialgebra}
A semi-algebra $A$ over a semi-field $K$ (or a $K$-semi-algebra) is a semi-vector
space $A$ over $K$ endowed with a binary operation called multiplication
of semi-vectors $\bullet: A \times A\longrightarrow A$ such that, $\forall \ u, v, w \in A$ and
$\lambda \in K$:
\begin{itemize}
\item [ $\operatorname{(1a)}$] $ u \bullet (v + w)= (u \bullet v) + (u \bullet w)$ (left-distributivity);
\item [ $\operatorname{(1b)}$] $ (u + v)\bullet w= (u \bullet w) + (v \bullet w)$ (right-distributivity);
\item [ $\operatorname{(2)}$] $ \lambda (u \bullet v)= (\lambda u)\bullet v = u \bullet (\lambda v)$.
\end{itemize}
\end{definition}

A semi-algebra $A$ is \emph{associative} if  $(u\bullet v)\bullet
w=u\bullet (v\bullet w)$ for all $u, v, w \in A$; $A$ is said to
be \emph{commutative} (or abelian) is the multiplication is commutative, that
is, $\forall \ u, v \in A$, $u\bullet v= v\bullet u$; $A$ is called
a semi-algebra with identity if there exists an element
$1_A \in A$ such that $\forall \ u \in A$, $1_A \bullet u = u \bullet 1_A =u$;
the element $1_A $ is called identity of $A$. The identity element
of a semi-algebra $A$ is unique (if exists). If $A$ is a
semi-free semi-vector space then the dimension of $A$ is its dimension
regarded as a semi-vector space. A semi-algebra is \emph{simple} if it
is simple as a semi-vector space.

\begin{example}\label{ex5}
The set ${\mathbb R}_{0}^{+}$ is a commutative semi-algebra with identity $e=1$.
\end{example}

\begin{example}\label{ex6}
The set of square matrices of order $n$ whose entries are in ${\mathbb R}_{0}^{+}$,
equipped with the sum of matrices, multiplication of a matrix by a scalar
(in ${\mathbb R}_{0}^{+}$, of course) and by multiplication of
matrices is an associative and non-commutative semi-algebra with identity $e=I_{n}$
(the identity matrix of order $n$), over ${\mathbb R}_{0}^{+}$.
\end{example}

\begin{example}\label{ex7}
The set ${\mathcal P}_{n}[x]$ of polynomials with coefficients
from ${\mathbb R}_{0}^{+}$ and degree less than or equal
to $n$, equipped with the usual of polynomial sum and scalar multiplication is a
semi-vector space.
\end{example}

\begin{example}\label{ex8}
Let $V$ be a semi-vector space over a semi-field $K$. Then the set
${\mathcal L}(V, V)=\{T:V\longrightarrow V;
T \operatorname{is \ a \ semi-linear \ operator}\}$ is a semi-vector space.
If we define a vector multiplication as the composite of semi-linear
operators (which is also semi-linear) then we have a semi-algebra
over $K$.
\end{example}

\begin{definition}\label{subsemialgebra}
Let $A$ be a semi-algebra over $K$. We say that a non-empty set $S \subseteq A$
is a semi-subalgebra if $S$ is closed under the operations of $A$, that is,
\begin{itemize}
\item [ $\operatorname{(1)}$] $\forall \ u, v \in A$, $u + v \in A$;
\item [ $\operatorname{(2)}$] $\forall \ u, v \in A$, $u \bullet v \in A$;
\item [ $\operatorname{(3)}$] $\forall \ \lambda \in K$ and $\forall u \in A$, $\lambda u \in A$.
\end{itemize}
\end{definition}

\begin{definition}\label{A-homomorphism}
Let $A$ and $B$ two semi-algebras over $K$. We say that a
map $T:A\longrightarrow B$ is an $K$-semi-algebra homomorphism
if, $\forall \ u, v \in A$ and $\lambda \in K$, the following conditions hold:
\begin{itemize}
\item [ $\operatorname{(1)}$] $T(u + v) = T(u) + T(v)$;
\item [ $\operatorname{(2)}$] $T(u \bullet v) = T(u) \bullet T(v)$;
\item [ $\operatorname{(3)}$] $T(\lambda v ) = \lambda T(v)$.
\end{itemize}
\end{definition}

Definition~\ref{A-homomorphism} means that $T$ is both a semi-ring homomorphism and also
semi-linear (as semi-vector space).

\begin{definition}\label{isomorphic}
Let $A$ and $B$ be two $K$-semi-algebras. A $K$-semi-algebra isomorphism
$T:A \longrightarrow B$ is a bijective $K$-semi-algebra homomorphism.
If there exists such an isomorphism, we say that $A$ is isomorphic to $B$, written $A\cong B$.
\end{definition}

The following results seems to be new, because semi-algebras over
${\mathbb R}_{0}^{+}$ are not much investigated in the literature.

\begin{proposition}\label{propalghomo}
Assume that $A$ and $B$ are two $K$-semi-algebras, where $K={\mathbb R}_{0}^{+}$
and $A$ has identity $1_A$. Let $T:A \longrightarrow B$
be a $K$-semi-algebra homomorphism. Then the following properties hold:
\begin{itemize}
\item [ $\operatorname{(1)}$] $T(0_A)= 0_B$;
\item [ $\operatorname{(2)}$] If $ u\in A$ is invertible then its inverse is
unique and $(u^{-1})^{-1}= u$;
\item [ $\operatorname{(3)}$] If $T$ is surjective then $T(1_A) = 1_B$, i.e., $B$ also has identity;
furthermore, $T(u^{-1})= [T(u)]^{-1}$;
\item [ $\operatorname{(4)}$] If $u, v \in A$ are invertible then
$(u\bullet v )^{-1}= v^{-1}\bullet u^{-1}$;
\item [ $\operatorname{(5)}$] the composite of $K$-semi-algebra homomorphisms is also a
$K$-semi-algebra homomorphism;
\item [ $\operatorname{(6)}$] if $T$ is a $K$-semi-algebra isomorphism then also
is $T^{-1}:B \longrightarrow A$.
\item [ $\operatorname{(7)}$] the relation $A \sim B$ if and only if $A$
is isomorphic to $B$ is an equivalence relation.
\end{itemize}
\end{proposition}
\begin{proof}
Note that Item~$\operatorname{(1)}$ holds because the additive cancelation
law holds in the definition of semi-vector spaces (see Definition\ref{defSVS}).
We only show Item $\operatorname{(3)}$ since the remaining items are direct.
Let $v \in B$; then there exists $u \in A$ such that $T(u)=v$. It then follows that
$v  \bullet T(1_A )= T(u\bullet 1_A)=v$ and $T(1_A ) \bullet v = T(1_A \bullet u)=v$;
which means that $T(1_A)$ is the identity of $B$, i.e., $T(1_A) = 1_B$.

We have: $T(u) \bullet T(u^{-1})= T( u \bullet u^{-1})=T(1_A)=1_B$ and
$T(u^{-1}) \bullet T(u)= T( u^{-1} \bullet u)=T(1_A)=1_B$, which implies
$T(u^{-1})= [T(u)]^{-1}$.
\end{proof}

\begin{proposition}\label{associunitsemi}
If $A$ is a $K$-semi-algebra with identity $1_A$ then $A$ can be embedded in
${\mathcal L}(A, A)$, the semi-algebra of semi-linear operators on $A$.
\end{proposition}
\begin{proof}
For every fixed $v \in A$, define $v^{*}:A \longrightarrow A$ as
$v^{*}(x) = v\bullet x$. It is easy to see that $v^{*}$
is a semi-linear operator on $A$. Define $h: A \longrightarrow {\mathcal L}(A, A)$ by
$h(v)= v^{*}$. We must show that
$h$ is a injective $K$-semi-algebra homomorphism where the product in
${\mathcal L}(A, A)$ is the composite of maps from $A$ into $A$.
Fixing $u, v \in A$, we have: $[h(u + v)](x)=
(u + v)^{*}(x)= (u + v)\bullet x = u\bullet x + v \bullet x =
u^{*}(x) + v^{*}(x) = [h(u)](x) + [h(v)](x)$,
hence $h(u + v)= h(u) + h(v)$. For $\lambda \in K$ and $v \in A$, it follows
that $[h(\lambda v)](x) = (\lambda v)^{*}(x)= (\lambda v)x = \lambda (vx)=
[\lambda h(v)](x)$, i.e., $h(\lambda v)= \lambda h(v)$. For fixed $u, v \in A$,
$[h(u\bullet v)](x)= (u\bullet v)^{*}(x)= (u\bullet v)\bullet x = u\bullet
(v\bullet x)=u\bullet v^{*}(x)=u^{*}(v^{*}(x))=[h(u)\circ h(v)](x)$, i.e.,
$h(u\bullet v)= h(u) \circ h(v)$.
Assume that $h(u)=h(v)$, that is, $u^{*}=v^{*}$; hence, for every $x \in A$,
$u^{*}(x) = v^{*}(x)$, i.e., $u\bullet x = v\bullet x$ . Taking in particular
$x=1_A$, it follows that $u = v$, which implies that $h$ is injective. Therefore,
$A$ is isomorphic to $h(A)$, where $h(A)\subseteq {\mathcal L}(A, A)$.
\end{proof}


\begin{definition}\label{semi-Liesemialgebra}
Let $A$ be a semi-vector space over a semi-field $K$. Then $A$ is
said to be a Lie semi-algebra if $A$ is equipped with
a product $[ \ , \ ]: A \times A\longrightarrow A$ such that the following conditions hold:
\begin{itemize}
\item [ $\operatorname{(1)}$] $[ \ , \ ]$ is semi-bilinear, i.e.,
fixing the first (second) variable, $[ \ , \ ]$ is semi-linear w.r.t.
the second (first) one;

\item [ $\operatorname{(2)}$] $[ \ , \ ]$ is anti-symmetric, i.e.,
$[v , v]=0$  $\forall \ v \in A$;

\item [ $\operatorname{(3)}$] $[ \ , \ ]$ satisfies the Jacobi identity:
$\forall \ u, v, w \in A$, $[u, [v,w]]+ [w, [u,v]]+ [v, [w, u]]=0$
\end{itemize}
\end{definition}

From Definition~\ref{semi-Liesemialgebra} we can see that a Lie semi-algebra
can be non-associa-tive, i.e., the product $[ \ , \ ]$ is not always associative.

Let us now consider the semi-algebra ${ \mathcal M}_n ({\mathbb R}_{0}^{+})$
of matrices of order $n$ with entries in ${\mathbb R}_{0}^{+}$
(see Example~\ref{ex6}). We know that ${ \mathcal M}_n ({\mathbb R}_{0}^{+})$ is simple,
i.e., with exception of the zero matrix (zero vector), no matrix
has (additive) symmetric. Therefore, the product of such matrices
can be nonzero. However, in the case of a Lie semi-algebra $A$, if $A$ is
simple then the unique product $[ \ , \ ]$ that can be defined over $A$ is
the zero product, as it is shown in the next result.

\begin{proposition}\label{semi-Lieabelian}
If $A$ is a simple Lie semi-algebra over a semi-field $K$ then the
semi-algebra is abelian, i.e., $[u, v]=0$ for all $u, v \in A$.
\end{proposition}
\begin{proof}
Assume that $u, v \in A$ and $[u, v ] \neq 0$. From
Items~$\operatorname{(1)}$~and~$\operatorname{(2)}$
of Definition~\ref{semi-Liesemialgebra}, it follows that $[u+v , u+v ]
=[u, u] + [u, v] + [v, u] + [v, v]=0$, i.e., $[u, v] + [v, u]=0$.
This means that $[u, v]$ has symmetric $[v, u]\neq 0$, a
contradiction.
\end{proof}

\begin{definition}\label{subLiesemi}
Let $A$ be a Lie semi-algebra over a semi-field $K$. A Lie semi-subalgebra
$B \subseteq A$ is a semi-subspace of $A$ which is closed under
$[u, v ]$, i.e., for all $u, v \in B$, $[u, v] \in B$.
\end{definition}

\begin{corollary}
All semi-subspaces of $A$ are semi-subalgebras of $A$.
\end{corollary}
\begin{proof}
Apply Proposition~\ref{semi-Lieabelian}.
\end{proof}

\section{Fuzzy Set Theory and Semi-Algebras}\label{sec3a}

The theory of semi-vector spaces and semi-algebras is a natural
generalization of the corresponding theories of vector spaces
and algebras. Since the scalars are in semi-fields (weak semi-fields), some
standard properties does not hold in this new context. However,
as we have shown in Section~\ref{sec3}, even in case of
nonexistence of symmetrizable elements, several results are still true.
An application of the theory of semi-vector spaces is in the investigation
on Fuzzy Set Theory, which was introduced by Lotfali Askar-Zadeh \cite{Zadeh:1965}.
In fact, such a theory fits in the investigation/extension
of results concerning fuzzy sets and their corresponding theory.
Let us see an example.

Let $L$ be a linearly ordered complete lattice with distinct smallest and
largest elements $0$ and $1$. Recall that a fuzzy number is a function
$x:{\mathbb R}\longrightarrow L$ on the field of real numbers satisfying the following
items (see \cite[Sect. 1.1]{Gahler:1999}): $\operatorname{(1)}$ for each
$\alpha \in L_0$ the set $x_{\alpha}= \{\varphi \in {\mathbb R} |
\alpha \leq x(\varphi)\} $ is a closed interval $[x_{\alpha l} ,
x_{\alpha r}]$, where $L_0= \{ \alpha \in L | \alpha > 0\}$;
$\operatorname{(2)}$ $\{\varphi \in {\mathbb R} | 0 < x(\varphi)\}$
is bounded.

We denote the set ${\mathbb R}_L$ to be the set of all fuzzy numbers;
${\mathbb R}_L$ can be equipped with a partial order in
the following manner: $x \leq y $ if and only if $x_{\alpha l} \leq y_{\alpha l}$
and $x_{\alpha r} \leq y_{\alpha r}$ for all $\alpha \in L_0$. In this scenario,
Gahler et al. showed that the concepts of semi-algebras can be utilized to
extend the concept of fuzzy numbers, according to the following proposition:
\begin{proposition}\cite[Proposition 19]{Gahler:1999}
The set ${\mathbb R}_L$ is an ordered commutative semi-algebra.
\end{proposition}
Thus, a direct utilization of the investigation of the structures of
semi-vector spaces and semi-algebras is the possibility to generate
new interesting results on the Fuzzy Set Theory.

Another work relating semi-vector spaces and Fuzzy Set Theory is the
paper by Bedregal et al. \cite{Milfont:2021}. In order to study
the aggregation functions (geometric mean, weighted average, ordered
weighted averaging, among others) w.r.t. an admissible order
(a total order $\preceq$ on $L_n ([0, 1])$ such that for all $x, y \in L_n ([0, 1])$,
$x \ {\leq}_{n}^{p} \ y \Longrightarrow x\preceq y$), the authors worked with
semi-vector spaces over a weak semi-field.

Let $L_n ([0, 1]) = \{(x_1, x_2 , \ldots , x_n ) \in  {[0, 1]}^{n}
| x_1 \leq x_2 \leq \ldots \leq x_n \}$ and $U= ([0, 1],
\oplus , \cdot)$ be a weak semi-field defined as follows: for all $x, y
\in [0, 1]$, $x \oplus y = \min\{ 1, x+y\}$ and $\cdot$ is the
usual multiplication. The product order proposed by Shang
et al.~\cite{Shang:2010} is given as follows: for all $x=
\{(x_1, x_2 , \ldots , x_n )$ and $y= \{(y_1, y_2 , \ldots ,
y_n )$ vectors in $L_n ([0, 1])$, define
$x \ {\leq}_{n}^{p} \ y \Longleftrightarrow {\pi}_{i}(x)\leq {\pi}_{i}(x) $
for each $i \in \{1, 2, \ldots , n\}$,
where ${\pi}_i : L_n ([0, 1]) \longrightarrow [0, 1] $ is the
$i$-th projection ${\pi}_i (x_1 , x_2 , \ldots , x_n ) = x_i$.
With these concepts in mind, the authors showed two important results:

\begin{theorem}(see \cite[Theorem 1]{Milfont:2021})\label{mil21}
${\mathcal L}_{n} ([0, 1]) = (L_n ([0, 1], \dotplus, \odot)$ is a
semi-vector space over $U$, where $r \odot v = (rx_1 ,
\ldots , rx_n )$ and  $u \dotplus v = (x_1 \oplus y_1 , \ldots , x_n
\oplus y_n ) $. Moreover,  $({\mathcal L}_{n} ([0, 1]),
{\leq}_{n}^{p})$ is an ordered semi-vector space over $U$,
where ${\leq}_{n}^{p}$ is the product order.
\end{theorem}

\begin{proposition}(see \cite[Propostion 2]{Milfont:2021})
For any bijection $f: \{1, 2 , \ldots , n\}
\longrightarrow \{1, 2 , \ldots , n\}$,
the pair\\ $({\mathcal L}_{n}([0, 1]), {\preceq}_f)$
is an ordered semi-vector space over $U$, where
${\preceq}_f$, defined in
\cite[Example 1]{Milfont:2021}, is an admissible order.
\end{proposition}
As a consequence of the investigation made, the authors propose an
algorithm to perform a multi-criteria and multi-expert decision making method.

Summarizing the ideas: the better the theory of semi-vector spaces
is extended and developed, the more applications and more results
we will have in the Fuzzy Set Theory. Therefore, it is important to understand
deeply which are the algebraic and geometry structures of semi-vector
spaces, providing, in this way, support for the development of the own
theory as well as other interesting theories as, for example, the
Fuzzy Set Theory.

\section{Summary}\label{sec4}

In this paper we have extended the theory of semi-vector spaces,
where the semi-field of scalars considered here is
the nonnegative real numbers. We have proved several results in the context
of semi-vector spaces and semi-linear transformations. We
introduced the concept of eigenvalues and eigenvectors of a
semi-linear operator and of a matrix and shown how to compute it
in specific cases. Topological properties of semi-vector spaces
such as completeness and separability were also investigated.
We have exhibited interesting new families of semi-vector
spaces derived from semi-metric, semi-norm, semi-inner product,
among others. Additionally, some results concerning semi-algebras
were presented. The results presented in this paper can be possibly
utilized in the development and/or investigation of new properties of
fuzzy systems and also in the study of correlated areas of research.
\section*{Acknowledgment}

\end{document}